\documentclass[reqno, 12pt]{amsart}

\pdfoutput=1

\usepackage{enumerate}
\usepackage{latexsym}
\usepackage[centertags]{amsmath}
\usepackage{amsfonts}
\usepackage{amssymb}
\usepackage{amsthm}
\usepackage{newlfont}
\usepackage{graphics}
\usepackage{color}
\usepackage{float}
\usepackage{diagbox}
\usepackage{hyperref}
\usepackage{longtable}
\usepackage{rotating}
\usepackage{multirow}

\usepackage{extarrows}

\usepackage[sort,compress,numbers]{natbib}

\usepackage[utf8]{inputenc}

\newtheorem{thm}{Theorem}[section]
\newtheorem{cor}[thm]{Corollary}
\newtheorem{lem}[thm]{Lemma}
\newtheorem{prop}[thm]{Proposition}

\theoremstyle{mydefinition}

\theoremstyle{myremark}
\newtheorem{rem}[thm]{Remark}
\newtheorem{exa}[thm]{Example}
\newtheorem{Collection}[thm]{Collection}

\newtheorem{prob}[thm]{Open Problem}
\allowdisplaybreaks[4]

\title[Frobenius Number and Genus]{On the Frobenius Number and Genus of a Collection of Semigroups Generalizing Repunit Numerical Semigroups}

\author{Feihu Liu$^{1}$, Guoce Xin$^{2, *}$, Suting Ye$^{3}$ and Jingjing Yin$^{4}$}

\address{$^{1, 2, 3, 4}$School of Mathematical Sciences,  Capital Normal University,
 Beijing 100048,  PR China}
\email{$^1$\texttt{liufeihu7476@163.com}\ \& $^2$\texttt{guoce\_xin@163.com}\ \& $^3$\texttt{yesuting0203@163.com}\ \newline \newline \& $^4$\texttt{yinjingj@163.com}}
\date{March 4, 2025}
\thanks{$*$ This work was partially supported by NSFC(12071311).}
\begin{document}
\maketitle

\begin{abstract}
Let $A=(a_1, a_2, \ldots, a_n)$ be a sequence of relative prime positive integers with $a_i\geq 2$. The Frobenius number $F(A)$ is the largest integer not belonging to the numerical semigroup $\langle A\rangle$ generated by $A$. The genus $g(A)$ is the number of positive integer elements not in $\langle A\rangle$. The Frobenius problem is to determine $F(A)$ and $g(A)$ for a given sequence $A$. In this paper, we study the Frobenius problem of $A=\left(a,h_1a+b_1d,h_2a+b_2d,\ldots,h_ka+b_kd\right)$ with some restrictions. An innovation is that $d$ can be a negative integer. In particular, when
$A=\left(a,ba+d,b^2a+\frac{b^2-1}{b-1}d,\ldots,b^ka+\frac{b^k-1}{b-1}d\right)$, we obtain formulas for $F(A)$ and $g(A)$ when $a\geq k-1-\frac{d-1}{b-1}$. Our formulas simplify further for some special cases, such as Mersenne, Thabit, and repunit numerical semigroups. Finally, we partially solve an open problem for the Proth numerical semigroup.
\end{abstract}

\def\D{{\mathcal{D}}}

\noindent
\begin{small}
 \emph{Mathematic subject classification}: Primary 11D07; Secondary 11A67, 11B75, 20M14.
\end{small}

\noindent
\begin{small}
\end{small}

\section{Introduction}

Throughout this paper, $\mathbb{Z}$, $\mathbb{N}$, and $\mathbb{P}$ denote the set of all integers, non-negative integers, and positive integers, respectively. We need to introduce some basic definitions (for instance, see \cite{A.Assi,J.C.Rosales}).

A subset $S$ is a \emph{submonoid} of $\mathbb{N}$ if $S\subseteq \mathbb{N}$, $0\in S$, and $S$ is closed under the addition in $\mathbb{N}$. If $\mathbb{N}\setminus S$ is finite, we say that $S$ is a \emph{numerical semigroup}. For a sequence (or set) $A=(a_1,a_2,\ldots,a_n)$ with $a_i \in\mathbb{P}$, denote by $\langle A\rangle$ the submonoid of $\mathbb{N}$ generated by $A$, that is,
$$\langle A\rangle=\left\{\sum_{i=1}^n x_ia_i \mid n\in \mathbb{P}, x_i\in \mathbb{N}, 1\leq i\leq n \right\}.$$
In \cite[Lemma 2.1]{J.C.Rosales}, we can see that $\langle A\rangle$ is a numerical semigroup if and only if $\gcd(A)=1$.
If $\langle A\rangle$ is a numerical semigroup, then we say that $A$ is a \emph{system of generators} of $\langle A\rangle$. If no proper subsequence (or subset) of $A$ generates $\langle A\rangle$, then we say that $A$ is a \emph{minimal system of generators} of $\langle A\rangle$, and the cardinality of $A$ is the embedding dimension of $\langle A\rangle$, denoted by $e(A)$. Moreover, the minimal system of generators is finite and unique (see \cite{J.C.Rosales}).

If $\gcd(A)=1$, then we have the following definitions:
\begin{enumerate}
  \item \emph{The Frobenius number} $F(A)$: The greatest integer not belonging to $\langle A\rangle$.

  \item \emph{The genus (or Sylvester number)} $g(A)$: The cardinality of $\mathbb{N}\backslash \langle A\rangle$.

  \item \emph{The pseudo-Frobenius numbers} $u$: If $u \in \mathbb{Z}\backslash \langle A\rangle$ and $u+s\in \langle A\rangle$ for all $s \in \langle A\rangle\backslash \{0\}$.

  \item The set $PF(A)$: The set of pseudo-Frobenius numbers of $\langle A\rangle$ (see \cite{Rosales2002}).

  \item \emph{The type} $t(A)$ of $\langle A\rangle$: The cardinality of $PF(A)$.
\end{enumerate}

The Frobenius number $F(A)$ has been widely studied. For $A=(a_1,a_2)$, Sylvester \cite{J. J. Sylvester1} obtained $F(A)=a_1a_2-a_1-a_2$ in 1882. For $n\geq 3$, Curtis \cite{F.Curtis} proved that the $F(A)$ cannot be given by closed formulas of a certain type. However, many special cases have been studied, such as arithmetic progression \cite{A. Brauer,Roberts1,E. S. Selmer}, geometric sequences \cite{D.C.Ong}, and triangular and tetrahedral sequences \cite{AMRobles}.
For more special sequences, see \cite{Ramrez Alfonsn,A. Tripathi1,A. Tripathi3,Liu-Xin,Fliuxin22,Rodseth}. Many special numerical semigroups are also considered, such as Fibonacci \cite{J.M.Marin}, Mersenne \cite{Rosales2016}, repunit \cite{Rosales.Repunit}, squares and cubes \cite{M.Lepilov}, Thabit \cite{Rosales2015}, and other numerical semigroups \cite{GuZe2020,KyunghwanSong2020,KyunghwanSong}.

The motivation of this paper comes from the repunit numerical semigroup: $\langle\{ \frac{b^{n+i}-1}{b-1}\mid i\in \mathbb{N}\}\rangle$. Its minimal system of generators is $S(b,n)=\left(\frac{b^{n}-1}{b-1}, \frac{b^{n+1}-1}{b-1},\ldots, \frac{b^{2n-1}-1}{b-1}\right)$ (see \cite{Rosales.Repunit}). The embedding dimension $e(S(b,n))$ is $n$.
In this paper, we consider the following more general collection of numerical semigroups.
\begin{Collection}\label{e-general-model}
Let $a,b,k\in\mathbb{P}, d\in\mathbb{Z}\backslash \{0\}$ with $a,b,k\geq 2$ and $\gcd(a,d)=1$. Our collection of numerical semigroups, called \emph{Collection \texttt{CNS}}, consists of all  $\langle A\rangle$ with
\begin{equation}
A=(a,Ha+dB)=\left(a,ba+d,b^2a+\frac{b^2-1}{b-1}d,\ldots,b^ka+\frac{b^k-1}{b-1}d\right),
\end{equation}
where we require $b^ia+\frac{b^i-1}{b-1}d>1$ for all $1\leq i\leq k$ if $d$ is a negative integer.
\end{Collection}

Here we have the embedding dimension $e(A)\leq k+1$. At this time, we can observe that
\begin{enumerate}
\item If $a=(2^m-1)\cdot 2^n-1$, $b=2$, $d=1$, and $k=n+m-1$, then $\langle A\rangle$ is a class of numerical semigroups as discussed in \cite{GuZeTang}.

\item If $a=(2^m+1)\cdot 2^n-(2^m-1)$, $b=2$, $d=2^m-1$, and $k\in\{n+1,n+m-1,n+m\}$, then $\langle A\rangle$ is a class of numerical semigroups as discussed in \cite{KyunghwanSong}.

\item If $a=\frac{b^n-1}{b-1}$, $b\geq 2$, $d=1$, and $k=n-1$, then $\langle A\rangle$ is the repunit numerical semigroups $S(b,n)$ in \cite{Rosales.Repunit}. In the case where $b=2$, this family specializes to the class of Mersenne numerical semigroups $S(n)$ considered in \cite{Rosales2016}.

\item If $a=b^{n+1}+\frac{b^n-1}{b-1}$, $b\geq 2$, $d=1$, and $k=n+1$, then $\langle A\rangle$ is a class of numerical semigroups in \cite{GuZe2020}.

\item If $a=(b+1)b^n-1$, $b\geq 2$, $d=b-1$, and $k=n+1$, then $\langle A\rangle$ is the Thabit numerical semigroups $T_{b,1}(n)$ of the first kind base $b$ in \cite{KyunghwanSong2020}. In the case where $b=2$, this family specializes to the class of Thabit numerical semigroups $T(n)$ considered in \cite{Rosales2015}.

\item If $d\in \mathbb{P}$, we do not have a condition on $k$. In the special case when $k=\min\{m-1\mid b^m-1>(b-1)(a-1)\}$, $\langle A\rangle$ specializes to the numerical semigroups $G_{b,d}(a)$ in  \cite{S.Ugolini}.

\item If $a=(2^r+1)2^n+1$, $b=2$, $d=-1$, and $k=n+r$, then $\langle A\rangle$ is a class of Proth numerical semigroups $P_{2^r+1}(n)$ in \cite{P.Srivastava}. In the case where $r=0$, this family specializes to the class of Cunningham numerical semigroups considered in \cite{KyunghwanSong2020}.
\end{enumerate}

The main purpose of this paper is to provide a unified approach to the above $7$ numerical semigroups, which are special cases of Collection \texttt{CNS}. Inspired by Collection \texttt{CNS}, we study the following more general collection.

\begin{Collection}\label{e-general-general-model}
Let $a,k\in\mathbb{P}$, $a,k\geq 2$, $d\in\mathbb{Z}\backslash \{0\}$, $\gcd(a,d)=1$, $b_1=1, b_{i+1}=s_ib_i+1$, $s_i\geq s_{i-1}$ for $1\leq i\leq k-1$ and $h_i=ub_i+1, u\in \mathbb{P}$ for $1\leq i\leq k$. Our general collection, called \emph{Collection \texttt{GCNS}}, consists of numerical semigroups $\langle A\rangle$ with
\begin{equation}
A=(a,Ha+dB)=(a,h_1a+db_1,h_2a+db_2,\ldots,h_ka+db_k),
\end{equation}
where we require $h_ia+db_i>1$ for $1\leq i\leq k$ if $d$ is a negative integer.
\end{Collection}

For this collection, we obtain formulas of $F(A)$ and $g(A)$ in Theorem \ref{a-2ad-22a3d}.
In some specializations, we can further obtain $PF(A)$ and $t(A)$.
This approach also applies to several other new numerical semigroups:
\begin{enumerate}
\item The case $b_i=\frac{b^i-1}{b-1}$, $u=b-1$, $b\geq 2$, $k=n$, $a=\frac{m(b^n-1)}{b-1}$, which reduces to repunit numerical semigroups when $m=1,d=1$;

\item The case $b_i=2^i-1$, $u=1$, $k=n+1, a=3\cdot 2^n-1$, which reduces to Thabit numerical semigroups $T(n)$ when $d=1$;

\item The case $b_i=2^i-1$, $u=1$, $a=m(2^k-1)+2^{k-1}-1$, $m\geq 1$, and $k\geq 3$;

\item The Proth numerical semigroups.
\end{enumerate}
We consider Proth numerical semigroups $P_m(n)$ and (partially) solve the following open problem in Section \ref{SecPorthNS}.

\begin{prob}{\em \cite[Section 7]{P.Srivastava}}\label{OpenQuestProthFP}
Let $m$ be an odd positive integer and $m<2^n$. The Proth numerical semigroup $P_m(n)$ is generated by
$\{m2^{n+i}+1\mid i\in \mathbb{N}\}$. Is there a formula to find the Frobenius number and other invariants of the Proth numerical semigroup $P_m(n)$ for arbitrary $m$?
\end{prob}

To compute $F(A)$ and $g(A)$ for the above Collection \texttt{GCNS}, we introduce a much simpler minimization problem $O_B^H(M)$, which simplifies as follows:
\begin{align*}
O_{B}^{H}(M)=\min\left\{uM+\sum_{i=1}^kx_i \bigm| \sum_{i=1}^k b_ix_i=M, \ x_i\in\mathbb{N}, 1\leq i\leq k\right\}.
\end{align*}
It turns out that $O_B^H(M)$ can be solved by the greedy algorithm. This is based on the fact that $B=\left(1,b_2,\ldots,b_k\right)$ is an orderly sequence, a concept that we will introduce in Section \ref{section2ord}.

The paper is organized as follows.
In Section 2, we provide some necessary lemmas and related results. Then we establish Theorem \ref{a-2ad-22a3d}, which gives formulas of the Frobenius number $F(A)$ and genus $g(A)$ when $A$ belongs to Collection \texttt{GCNS}, $ua+d+k-2\geq \sum_{i=1}^{k-1}s_i$, and $s_i\leq u+1$. As special cases, we obtain formulas for $F(A)$ and $g(A)$ when $\langle A\rangle$ belongs to Collection \texttt{CNS} and $a\ge k-1-\frac{d-1}{b-1}$.
Sections 3 and 4 focus on applications of Theorem \ref{a-2ad-22a3d} for some special $a$ and $k$.
Semigroups of this form are related to Mersenne, Thabit, and repunit numerical semigroups. The formulas simplify in these cases. In Section 5, we mainly study the Proth numerical semigroup and (partially) solve the Open Problem \ref{OpenQuestProthFP}. In Section 6, we will mention other numerical semigroups and consider the Frobenius problem for a new class of numerical semigroups.
Section 7 is a concluding remark.

\section{Collection \texttt{GCNS} with Some Restrictions}\label{section2ord}

\subsection{Crude Formula for Frobenius Numbers}

It is convenient to use the short-hand notation
$A:=(a, B)=(a, b_1, b_2, \ldots, b_k)$. Let $\gcd(A)=1$, $a,b_i\in \mathbb{P}$. The set
$$\langle A\rangle=\left\{ ax+\sum_{i=1}^kb_ix_i\ \mid x, x_i\in \mathbb{N}\right\}$$
is a numerical semigroup.
Now we define the following order relation in $\mathbb{Z}$: $a\preceq_{\langle A\rangle} b$ if $b-a \in \langle A\rangle$. It is proved that this relation $\preceq_{\langle A\rangle}$ is a partial order in \cite{J.C.Rosales}.

Let $w \in \langle A\rangle\backslash \{0\}$.
The \emph{Ap\'ery set} of $w$ in $\langle A\rangle$ is $Ape(A,w)=\{s\in \langle A\rangle \mid s-w\notin \langle A\rangle\}$. In \cite{J.C.Rosales}, it is shown that
$$Ape(A,w)=\{N_0,N_1,N_2,\ldots,N_{w-1}\},$$
where $N_r:=\min\{ a_0\mid a_0\equiv r\mod w, \ a_0\in \langle A\rangle\}$, $0\leq r\leq w-1$. We usually take $w:=a$.

Brauer and Shockley \cite{J. E. Shockley} and Selmer \cite{E. S. Selmer} gave the following results, respectively.

\begin{lem}[\cite{J. E. Shockley}, \cite{E. S. Selmer}]\label{LiuXin001}
Suppose $A=(a, B)=(a, b_1, b_2, \ldots, b_k)$. The \emph{Ap\'ery set} of $a$ in $\langle A\rangle$ is
$Ape(A,a)=\{N_0,N_1,N_2,\ldots,N_{a-1}\}$. Then the Frobenius number and genus of $A$ are, respectively:
\begin{align*}
F(A)&=F(a, B)=\max_{r\in \lbrace 0, 1, \ldots, a-1\rbrace}N_r -a,\\
g(A)&=g(a, B)=\frac{1}{a}\sum_{r=1}^{a-1}N_r-\frac{a-1}{2}.
\end{align*}
\end{lem}

\begin{lem}{\em \cite[Proposition 2.20]{J.C.Rosales}}\label{Pseudo-FP-Prop}
Let $\langle A\rangle$ be a numerical semigroup. Then
$$PF(A)=\left\{w-a\mid w \in \max\nolimits_{\preceq_{\langle A\rangle}} Ape(A, a) \right\},$$
where $Ape(A,a)=\{N_0,N_1,\ldots,N_{a-1}\}$.
\end{lem}

The following result easily follows from the definition of $N_r$ with $r\in \{0,\ldots,a-1\}$.
\begin{prop}\label{0201}
Let $A=(a, b_1, \ldots, b_k)$, $\gcd(a, d)=1, \ d\in \mathbb{Z}\backslash \{0\}$. Then we have
\begin{equation}
\{N_0, N_1, N_2,\ldots, N_{a-1}\}=\{N_{d\cdot 0}, N_{d\cdot 1}, N_{d\cdot 2},\ldots, N_{d\cdot (a-1)}\}.\label{0402}
\end{equation}
\end{prop}

We observe that the argument in \cite{Liu-Xin} for $A=(a,a+dB)$ naturally extends to the general case of
$A=(a, Ha+dB)=(a, h_1a+db_1, h_2a+db_2,\ldots, h_ka+db_k)$ with $\gcd(a,d)=1$.
We find $F(A)$ and $g(A)$ are closely related to a minimization problem defined by:
$$O_{B}^{H}(M):=\min\left\{\sum_{i=1}^kh_ix_i \mid \sum_{i=1}^k b_ix_i=M, \ x_i\in\mathbb{N}, 1\leq i\leq k\right\}.$$
It reduces to the $O_B(M)$ in \cite{Liu-Xin} when $H=(1,1,\dots,1)$.

For convenience, henceforth we shall always assume $x_i\in\mathbb{N}, 1\leq i\leq k$ unless specified otherwise.
\begin{lem}\label{0202}
Let $A=(a, h_1a+db_1, \ldots, h_ka+db_k)$, $a, k, h\in\mathbb{P}$, $a\geq 2$, $d\in \mathbb{Z}\backslash \{0\}$ and $\gcd(A)=1$, $m\in\mathbb{N}$,  $\gcd(a, d)=1$. If $d$ is a negative integer, then we require $h_ia+db_i>1$ for $1\leq i\leq k$.
For a given $ 0\leq r\leq a-1$, we have
\begin{equation}\label{0203}
N_{dr}=\min \left\{O_{B}^{H}(ma+r) \cdot a+(ma+r)d \mid m\in \mathbb{N}\right\}.
\end{equation}
\end{lem}
\begin{proof}
We obtain the following equalities:
\begin{align*}
N_{dr}&=\min\{ a_0\mid a_0\equiv dr\mod a;\ a_0\in \langle A\rangle\}
\\&=\min\left\{\sum_{i=1}^k(h_ia+db_i)x_i\mid \sum_{i=1}^k(h_ia+db_i)x_i\equiv dr\mod a, \ x_i\in\mathbb{N}, 1\leq i\leq k \right\}
\\&=\min\left\{\left(\sum_{i=1}^kh_ix_i\right)\cdot a+d\cdot\sum_{i=1}^kb_ix_i\mid d\sum_{i=1}^kb_ix_i\equiv dr\mod a, \ x_i\in\mathbb{N}, 1\leq i\leq k \right\}
\\&=\min\left\{\left(\sum_{i=1}^kh_ix_i\right)\cdot a+d\cdot\sum_{i=1}^kb_ix_i\mid \sum_{i=1}^kb_ix_i\equiv r\mod a, \ x_i\in\mathbb{N}, 1\leq i\leq k \right\}
\\&=\min\left\{\left(\sum_{i=1}^kh_ix_i\right)\cdot a+d(ma+r)\mid \sum_{i=1}^kb_ix_i=ma+r, \ m, x_i\in\mathbb{N}, 1\leq i\leq k \right\}.
\end{align*}
Now for fixed $m$, and hence fixed $M=ma+r$,
$\sum_{i=1}^kh_ix_i$ is minimized to $O_{B}^H(ma+r)$. This completes the proof.
\end{proof}

By Lemma \ref{0202}, we can define an intermediate function with respect to $m\in \mathbb{N}$, namely:
$$N_{dr}(m):=O_{B}^{H}(ma+r) \cdot a+(ma+r)d.$$
In our cases, it is not hard to show that $N_{dr}(m)$ is increasing with respect to $m$, so that $N_{dr}=N_{dr}(0)$. This allows us to
further obtain the formulas of $F(A)$ and $g(A)$.

Before proceeding further, it is necessary to introduce the following definition.
For a given positive integer sequence $B=(b_1, b_2, \ldots, b_k)$, $1=b_1<b_2<\cdots <b_k$ and $M\in \mathbb{N}$, let
\begin{equation}\label{hahaha4}
opt_B(M):=O_B(M)=\min\left\{\sum_{i=1}^kx_i \mid \sum_{i=1}^kb_ix_i=M, \ \ M,x_i\in \mathbb{N}, 1\leq i\leq k\right\}.
\end{equation}
The problem $opt_B(M)$ is called \emph{the change-making problem} \cite{AnnAdamaszek}. A strategy is called greedy when it uses as many of the maximum values as possible, followed by as many of the next highest values as possible, and so on. 
We denote by $grd_B(M)$ the number of elements used in $B$ by the greedy strategy. Then we have $opt_B(M)\leq grd_B(M)$. If the greedy solution is always optimal, i.e., $opt_B(M)=grd_B(M)$ for all $M>0$, then we call the sequence $B$ \emph{orderly}; otherwise, we call it \emph{non-orderly}. For example, the sequence $B=(1,5,16)$ is non-orderly. Because $20=16+1+1+1+1=5+5+5+5$, we have $opt_B(20)=4<5=grd_B(20)$.

In general, it is hard to determine if a sequence is orderly. The following lemma gives a nice sufficient condition.
\begin{lem}[One-Point Theorem, \cite{LJCowen,TCHu,MJMagazine}]\label{one-point}
Suppose $B^{\prime}=(1,b_1,\ldots,b_k)$ is orderly and $b_{k+1}>b_k$. Let $s=\lceil b_{k+1} / b_k\rceil$. Then the sequence $B=(1,b_1,\ldots,b_k,b_{k+1})$ is orderly if and only if $opt_B(sb_k)=grd_B(sb_k)$.
\end{lem}

Using the \emph{One-Point Theorem}, we construct a class of orderly sequences as follows.
\begin{lem}\label{B137ordely}
Let $k\in \mathbb{P}$, $b_1=1$, $b_{i+1}=s_ib_i+1$, and $s_i\geq s_{i-1}$ for $1\leq i\leq k-1$. Then the sequence $B=\left(b_1,b_2,\ldots,b_k\right)$ is orderly, i.e., $opt_B(M)=grd_B(M)$ for all $M\in \mathbb{P}$.
\end{lem}
\begin{proof}
We proceed by induction on $k$ to establish the result.
The lemma clearly holds for $k\leq 2$, as the sequences $B=(1)$ and $B=\left(1,b_2\right)$ are orderly. Suppose $\left(1,b_2,\ldots,b_{k-1}\right)$ is orderly. By assumption, $\left\lceil \frac{b_{k}}{b_{k-1}}\right\rceil=s_{k-1}+1$. By $\left\lceil \frac{b_{k}}{b_{k-1}}\right\rceil\cdot b_{k-1}=(s_{k-1}+1)b_{k-1}=b_{k}+b_{k-1}-1=b_{k}+s_{k-2}b_{k-2}$ and $s_{k-2}+1\leq s_{k-1}+1$, we have  $opt_B\left(\left\lceil \frac{b_{k}}{b_{k-1}}\right\rceil\cdot b_{k-1}\right)=grd_B\left(\left\lceil \frac{b_{k}}{b_{k-1}}\right\rceil\cdot b_{k-1}\right)$. By Lemma \ref{one-point}, the sequence $B$ is orderly. This completes the proof.
\end{proof}

For the sequence $B$ as in Lemma \ref{B137ordely}, we can further study the Frobenius problem for
$A=(a,Ha+dB)=\left(a,h_1a+b_1d,h_2a+b_2d,\ldots,h_ka+b_kd\right)$.

\begin{prop}\label{greedproper}
Suppose $B=\left(b_1,b_2,\ldots,b_k\right)$, with $b_1=1$, $b_{i+1}=s_ib_i+1$, and $s_i\geq s_{i-1}$ for $1\leq i\leq k-1$.
Then, for any $M\in \mathbb{P}$, $grd_B(M)=(x_1,\dots, x_k)$ is uniquely characterized by the following three properties.
\begin{enumerate}
  \item $x_k=\left\lfloor \frac{M}{b_k}\right\rfloor$.

  \item $x_i\in \{0,1,\ldots,s_i\}$ for every $1\leq i\leq k-1$.

  \item if $2\leq i\leq k-1$ and $x_i=s_i$, then $x_1=\cdots =x_{i-1}=0$.
\end{enumerate}
\end{prop}
\begin{proof}
By $b_{i+1}=s_ib_i+1$ for $1\leq i\leq k-1$ and Lemma \ref{B137ordely}, the proof is obvious.
\end{proof}

If a solution $X=(x_1,x_2,\ldots,x_k)$ of $\sum_i x_ib_i=M$ satisfies the above conditions (1), (2), and (3), then we call $X=X(M)$ the \emph{greedy presentation} of $M$.
Define $R(M)=\{X(m): 0\leq m \leq M\}$, and define a colexicographic order on $R(M)$ as follows:
\begin{align*}
(x_1^{\prime},x_2^{\prime},\ldots,x_k^{\prime})\preceq (x_1,x_2,\ldots,x_k)
\Longleftrightarrow & x_i^{\prime}=x_i\ \ \text{for any}\ \ i>0\ \ \text{or}
\\& x_j^{\prime}<x_j, x_i^{\prime}=x_i\ \ \text{for a certain}\ \ j>0\ \ \text{and any}\ \ i>j.
\end{align*}
Obviously, the order relation $\preceq$ is a total order on $R(M)$.

\begin{rem}
Lemma \ref{B137ordely} essentially provides the construction process of the greedy presentations $X=X(m)=(x_1,x_2,\ldots,x_k)$ in $R(M)$.
\end{rem}

\begin{exa}
Let $B=(1,3,7,15)$ and $M=23$. The elements in $R(M)$ are listed as follows.
\begin{align*}
&X(0)=(0,0,0,0),\ \ \ \ X(1)=(1,0,0,0),\ \ \ \ X(2)=(2,0,0,0),\ \ \ \ X(3)=(0,1,0,0),
\\&X(4)=(1,1,0,0),\ \ \ \ X(5)=(2,1,0,0),\ \ \ \ X(6)=(0,2,0,0),\ \ \ \ X(7)=(0,0,1,0),
\\&X(8)=(1,0,1,0),\ \ \ \ X(9)=(2,0,1,0),\ \ \ X(10)=(0,1,1,0),\ \ X(11)=(1,1,1,0),
\\&X(12)=(2,1,1,0),\ \ X(13)=(0,2,1,0),\ \ X(14)=(0,0,2,0),\ \ X(15)=(0,0,0,1),
\\&X(16)=(1,0,0,1),\ \ X(17)=(2,0,0,1),\ \ X(18)=(0,1,0,1),\ \ X(19)=(1,1,0,1),
\\&X(20)=(2,1,0,1),\ \ X(21)=(0,2,0,1),\ \ X(22)=(0,0,1,1),\ \ X(23)=(1,0,1,1).
\end{align*}
\end{exa}

\subsection{The Frobenius Problem}

Now we consider the $O_B^H(M)$ for $B=\left(b_1,b_2,\ldots,b_k\right)$, $H=(h_1,h_2,\ldots,h_k)$, and any $M\in \mathbb{P}$, where
\begin{align}
& b_1=1,\ b_{i+1}=s_ib_i+1\ \ \text{and}\ \ s_i\geq s_{i-1}\ \ \text{for}\ \ 1\leq i\leq k-1,\label{BBcondition}
\\& h_i=ub_i+1, u\in \mathbb{P}\ \ \text{for}\ \ 1\leq i\leq k.\label{HHcondition}
\end{align}
We obtain
\begin{align*}
O_{B}^{H}(M)&=\min\left\{\sum_{i=1}^k(ub_i+1)x_i \mid \sum_{i=1}^k b_ix_i=M, \ x_i\in\mathbb{N}, 1\leq i\leq k\right\}
\\&=\min\left\{uM+\sum_{i=1}^kx_i \mid \sum_{i=1}^k b_ix_i=M, \ x_i\in\mathbb{N}, 1\leq i\leq k\right\}\\
&=uM+opt_B(M).
\end{align*}

For Collection \texttt{GCNS} with $\gcd(a,d)=1$, we have
$\gcd(a,h_1a+db_1)=\gcd(a,db_1)=1$. Furthermore, we have $\gcd(A)=1$. Therefore, $\langle A\rangle$ is a numerical semigroup. Now, we give a characterization of its Ap\'ery set.

\begin{thm}\label{HBuadNDr}
Let $A$ be in \emph{Collection \texttt{GCNS}}. If $ua+d+k-2\geq \sum_{i=1}^{k-1}s_i$, then $N_{dr}(m)$ is increasing with respect to $m\in \mathbb{N}$. More precisely, if $X(r)=(x_1,x_2,\ldots,x_k)$ is the greedy presentation of $r$, then
\begin{equation}\label{Ndr-pre}
N_{dr}=\sum_{i=1}^kx_ia+r(ua+d)=\left(\sum_{i=1}^k(ub_i+1)x_i\right)a+rd.
\end{equation}
\end{thm}
\begin{proof}
Recall that if $X(ma+r)=(y_1,y_2,\ldots,y_k)$, then
$$N_{dr}(m)=\left(u(ma+r)+\sum_{i=1}^k y_i\right)a+(ma+r)d.$$
Write $X((m+1)a+r)=(z_1,\ldots,z_k)$. Then we see that $z_{k}\geq y_k$, and
\begin{align*}
N_{dr}(m+1)-N_{dr}(m)&=ua^2+ad+\left(\sum_{i=1}^k z_i-\sum_{i=1}^k y_i\right)a
\\& \geq\left(ua+d+\sum_{i=1}^{k-1}z_i-\sum_{i=1}^{k-1}y_i\right)a
\\(\textrm{by\ Proposition\ \ref{greedproper}})\quad  & \geq\left(ua+d-\sum_{i=1}^{k-1}s_i+k-2\right)a\geq 0.
\end{align*}
Thus $N_{dr}(m)$ is increasing so that $N_{dr}=N_{dr}(0)$. This completes the proof.
\end{proof}

In Equation \eqref{Ndr-pre} with $X(r)=(x_1,x_2,\ldots,x_k)$, we call $w(r)=\sum_{i=1}^k(ub_i+1)x_i$ the \emph{weight} of $r$. Now we have the following result.

\begin{lem}\label{colex-incre}
Let $M\in \mathbb{N}$. The sequences $B$ and $H$ satisfy conditions \eqref{BBcondition} and \eqref{HHcondition}, respectively, and $s_i\leq u+1$ for $1\leq i\leq k-1$. Suppose $X(r_1)=(x_1^{\prime},\ldots,x_k^{\prime})\in R(M)$ and $X(r_2)=(x_1,\ldots,x_k)\in R(M)$. If $(x_1^{\prime},\ldots,x_k^{\prime})\preceq (x_1,\ldots,x_k)$, then $w(r_1)\leq w(r_2)$. In particular, if $r_1\neq r_2$ and $s_i<u+1$, then $w(r_1)<w(r_2)$.
\end{lem}
\begin{proof}
We consider the following two cases.

If $x_i^{\prime}=x_i$ for any $1\leq i\leq k$, then $w(r_1)=w(r_2)$.

If there exists a $1\leq t\leq k$ such that $x_t^{\prime}<x_t$ and $x_{t+1}^{\prime}=x_{t+1},\ldots,x_k^{\prime}=x_k$, then $x_t^{\prime}+1\leq x_t$ and
\begin{align*}
w(r_1)&=\sum_{i=1}^k(ub_i+1)x_i^{\prime}=\sum_{i=1}^{t}(ub_i+1)x_i^{\prime}
+\sum_{j=t+1}^{k}(ub_j+1)x_j^{\prime}
\\&\leq (ub_1+1)s_1+\sum_{i=2}^{t-1}(ub_i+1)(s_i-1)+(ub_t+1)x_t^{\prime}+\sum_{j=t+1}^{k}(ub_j+1)x_j
\\&=\sum_{i=1}^{t-2}\big((ub_i+1)s_i-ub_{i+1}-1\big)+(ub_{t-1}+1)s_{t-1}+(ub_t+1)x_t^{\prime}
+\sum_{j=t+1}^{k}(ub_j+1)x_j
\\&=\sum_{i=1}^{t-2}(s_i-u-1)+(ub_t-u+s_{t-1})+(ub_t+1)x_t^{\prime}+\sum_{j=t+1}^{k}(ub_j+1)x_j
\\&\leq \sum_{i=1}^{t-2}(s_i-u-1)+(ub_t+1)+(ub_t+1)x_t^{\prime}+\sum_{j=t+1}^{k}(ub_j+1)x_j
\\&\leq (ub_t+1)x_t+\sum_{j=t+1}^{k}(ub_j+1)x_j\leq w(r_2).
\end{align*}
Now we consider the case $s_i<u+1$. If $t=1$ and $x_1^{\prime}<x_1$, then obviously $w(r_1)<w(r_2)$. If $t>1$ and $x_t^{\prime}<x_t$, then we have $ub_t-u+s_{t-1}<ub_t+1$ in the above formula. Thus, we get $w(r_1)<w(r_2)$.
This completes the proof.
\end{proof}

Now we can obtain the main result of this section.
\begin{thm}\label{a-2ad-22a3d}
Let $A$ be in \emph{Collection \texttt{GCNS}}. Let $X(a-1)=(x_1,x_2,\ldots,x_k)$ be the greedy presentation of $a-1$.
Obviously $x_k=\left\lfloor \frac{a-1}{b_k}\right\rfloor$. If $ua+d+k-2\geq \sum_{i=1}^{k-1}s_i$, $a+d\geq 0$, and $s_i< u+1$ for $1\leq i\leq k-1$, then we have
\begin{align*}
F(A)=\sum_{i=1}^kx_i\cdot a+(a-1)(ua+d)-a.
\end{align*}
Specifically, if $d$ is a positive integer, then we only need $ua+d+k-2\geq \sum_{i=1}^{k-1}s_i$ and $s_i\leq u+1$ for $1\leq i\leq k-1$. The above equation still holds.
\end{thm}
\begin{proof}
By Theorem \ref{HBuadNDr}, Lemma \ref{colex-incre}, and Equation \eqref{Ndr-pre}, when $s_i<u+1$, it follows that
$$N_{d(r+1)}-N_{dr}=(w(r+1)-w(r))a+d\geq a+d\geq0.$$
Therefore, we have
$$\max N_{dr}=N_{d(a-1)}=\sum_{i=1}^kx_i\cdot a+(a-1)(ua+d).$$
Further from Lemma \ref{LiuXin001}, we can obtain the Frobenius number formula $F(A)$. In particular, if $d$ is a positive integer, then we only need the constraint $s_i\leq u+1$ by Lemma \ref{colex-incre}. This completes the proof.
\end{proof}

\begin{thm}\label{a-2ad-22a3d111111}
Let $A$ be in \emph{Collection \texttt{GCNS}}. If $ua+d+k-2\geq \sum_{i=1}^{k-1}s_i$, then we have
\begin{align*}
g(A)=\sum_{r=1}^{a-1}\left(\sum_{i=1}^kx_i\right)_r+\frac{(a-1)(ua+d-1)}{2},
\end{align*}
where $\left(\sum_{i=1}^{k}x_i\right)_r$ denotes the sum of elements in the greedy presentation of $r$. Obviously $x_k=\left\lfloor \frac{r}{b_k}\right\rfloor$.
\end{thm}
\begin{proof}
Using Equation \eqref{Ndr-pre} and Lemma \ref{LiuXin001}, we can obtain the genus $g(A)$. This completes the proof.
\end{proof}

\begin{exa}
Let $B=(1,3,7,29)$ and $u=3$. We get $s_1=2,s_2=2,s_3=4$ and $H=(4,10,22,88)$. Therefore, we have
$A=(a,4a+d,10a+3d,22a+7d,88a+29d)$. For $a=21$ and $d=2$, the greedy presentation of $a-1$ is $(0,2,2,0)$, and the Frobenius number is $F(A)=1363$. Similarly, we can obtain the genus $g(A)=694$. We can use the \emph{numericalsgps GAP} package \cite{M.Delgado} to verify the correctness of the above results.
\end{exa}

\begin{exa}
Let $B=(1,3,7,15)$ and $u=2$. We get $s_1=2,s_2=2,s_3=2$ and $H=(3,7,15,31)$. Therefore, we have
$A=(a,3a+d,7a+3d,15a+7d,31a+15d)$. For $a=25$ and $d=-2$, the greedy presentation of $a-1$ is $(2,0,1,1)$, and the Frobenius number is $F(A)=1251$. Similarly, we can obtain the genus $g(A)=622$. If $a=25$ and $d=11$, then we have $F(A)=1539$ and $g(A)=778$.
\end{exa}

From Theorem \ref{a-2ad-22a3d}, we can easily obtain the following results.
\begin{cor}\label{HaHeLiuCor}
Let $A$ be in \emph{Collection \texttt{CNS}} with $d\in\mathbb{P}$. Let $X(a-1)=(x_1,x_2,\ldots,x_k)$ be the greedy presentation of $a-1$. If $a\geq k-1-\frac{d-1}{b-1}$, then we have
\begin{align*}
F(A)=\left((b-1)a-b+d+\sum_{i=1}^kx_i\right)a-d.
\end{align*}
\end{cor}
\begin{proof}
Note that $d$ is a positive integer. The corollary follows by applying Theorem \ref{a-2ad-22a3d} to the special case $s_1=s_2=\cdots =s_{k-1}=b$ and $u=b-1$.
\end{proof}

\begin{cor}\label{HaHeLiuCor111111}
Let $A$ be in \emph{Collection \texttt{CNS}}. If $a\geq k-1-\frac{d-1}{b-1}$, then we have
\begin{align*}
g(A)=\sum_{r=1}^{a-1}\left(\sum_{i=1}^kx_i\right)_r+\frac{(a-1)((b-1)a+d-1)}{2},
\end{align*}
where $\left(\sum_{i=1}^{k}x_i\right)_r$ denotes the sum of elements in the greedy presentation of $r$.
\end{cor}
\begin{proof}
By applying Theorem \ref{a-2ad-22a3d111111} to the special case where $s_1=s_2=\cdots =s_{k-1}=b$ and $u=b-1$, we obtain the result.
\end{proof}

Let's give the following two examples. The proof is left to the readers.
\begin{exa}
Let $A=(a, 2a+d, 4a+3d)$, $a,d\in \mathbb{P}$, $\gcd(a,d)=1$, and $a\geq 2$. We have the Frobenius number
\begin{align*}
&F(A)=2a^2-\left(3-d+2\left\lfloor \frac{a-1}{3}\right\rfloor\right)a-d.
\end{align*}
The genus is as follows:
$$\begin{aligned}
  g(A)=
\left\{
    \begin{array}{lc}
    \frac{(a-1)(2a+d-1)}{2}-\frac{a(a-3)}{3}&\ \text{if}\ \ a\equiv 0 \mod 3;\ \ \ \ \\
    \frac{(a-1)(2a+d-1)}{2}-\frac{(a-1)(a-2)}{3}& \text{if}\ \ a\equiv 1,2 \mod 3.
    \end{array}
\right.
\end{aligned}$$
\end{exa}

\begin{exa}
Let $A=(a, 2a+d, 4a+3d, 8a+7d)$, $a,d\in \mathbb{P}$, $\gcd(a,d)=1$ and $a\geq 7$. We have the Frobenius number
$$F(A)=a^2+\left(\left\lfloor \frac{a-1}{7}\right\rfloor +d+v(a \mod 7)\right)a-d,$$
where $(v(r))_{0\leq r \leq 6} =(0,-2,-1,0,-1,0,1).$
The genus is as follows:
$$g(A)=\frac{(a-1)(a+d-1)}{2}+\frac{a^2+15a +c(a \mod 7)}{14},$$
where $(c(r))_{0\leq r \leq 6} =(0,-16,-20,-12,-20,-16,0).$
\end{exa}

\section{A Generalization of Repunit Numerical Semigroups}

In this section, we focus on Collection \texttt{GCNS} in the case $b_i=\frac{b^i-1}{b-1}$, $u=b-1$, $k=n$, $a=\frac{m(b^n-1)}{b-1}$, $n,b\geq 2$, $d\in \mathbb{P}$, and $m\geq 1$. That is, we consider
the numerical semigroup $\langle A\rangle$ generated by
$$A=\Bigg(\frac{m(b^n-1)}{b-1}, b\cdot \frac{m(b^n-1)}{b-1}+d, b^2\cdot \frac{m(b^n-1)}{b-1}+\frac{b^2-1}{b-1}d, \ldots, b^{n}\cdot \frac{m(b^n-1)}{b-1}+\frac{b^{n}-1}{b-1}d\Bigg).$$
This is a generalization of the repunit numerical semigroup. We solve the Frobenius problem by giving closed formulas for the  Frobenius number, genus, and pseudo-Frobenius numbers.

By Proposition \ref{greedproper} along with $a-1=\frac{m(b^n-1)}{b-1}-1=(m-1)\cdot \frac{b^{n}-1}{b-1}+b\cdot \frac{b^{n-1}-1}{b-1}$,
we have
\begin{align}\label{Ra-1repun}
R(a-1)&=\Big\{(x_1,x_2,\ldots,x_{n}) \mid 0\leq x_n\leq m-1; \ 0\leq x_i\leq b\ \ \text{for}\ \ 1\leq i\leq n-1;\notag
\\&\ \ \ \ \
\text{if}\ \ x_i=b,\ \ \text{then}\ \  x_j=0\ \  \text{for}\ \ j\leq i-1\leq n-2\Big\}.
\end{align}
It can be verified that $\#R(a-1)=\frac{m(b^n-1)}{b-1}=a$.

\begin{thm}\label{m2nd-repunit}
For $A=\left(a,ba+d,b^2a+\frac{b^2-1}{b-1}d,\ldots,b^ka+\frac{b^k-1}{b-1}d\right)$, assume $k=n$, $a=\frac{m(b^n-1)}{b-1}$, $b,n\geq 2$, $m\geq 1$, $d\in \mathbb{P}$, and $\gcd(a,d)=1$. We have
\begin{align*}
&F(A)=(mb^n+d-1)\cdot\frac{m(b^n-1)}{b-1}-d,
\\&g(A)=\frac{m(b^n-1)(mb^n+d-1)}{2(b-1)}+\frac{mb^n(n-2)+1-d}{2}.
\end{align*}
\end{thm}
\begin{proof}
By Corollary \ref{HaHeLiuCor} along with $a-1=(m-1)\cdot \frac{b^{n}-1}{b-1}+b\cdot\frac{b^{n-1}-1}{b-1}$, it follows that
\begin{align*}
F(A)&=\left((b-1)\cdot \frac{m(b^n-1)}{b-1}-b+d+(m-1+b)\right)\cdot \frac{m(b^n-1)}{b-1}-d
\\&=(mb^n+d-1)\cdot\frac{m(b^n-1)}{b-1}-d.
\end{align*}
Next, we analyze the genus $g(A)$. Based on the composition of $R(a-1)$ in Equation \eqref{Ra-1repun}, we obtain
\begin{align*}
\sum_{r=1}^{a-1}\left(\sum_{i=1}^{n} x_i\right)_r=&\sum_{i=0}^{m-1}i\cdot (b^{n-1}+b^{n-2}+\cdots +1)+m\left((n-1)b^{n-2}\cdot\sum_{i=0}^{b-1}i\right)
\\&\ +m\left(\sum_{j=1}^{n-1}\left(b\cdot b^{n-1-j}+(n-1-j)\cdot b^{n-2-j}\cdot\sum_{i=0}^{b-1}i\right)\right)
\\=&\frac{m(m-1)(b^n-1)}{2(b-1)}+\frac{m(n-1)b^{n-1}(b-1)}{2}
\\&\ +\sum_{j=1}^{n-1}m\left(b^{n-j}+\frac{b^{n-1-j}(b-1)(n-j-1)}{2}\right)
\\=&\frac{m(m-1)(b^n-1)}{2(b-1)}+\frac{m(b^n-b)}{2(b-1)}+\frac{mb^n(n-1)}{2}.
\end{align*}
By Corollary \ref{HaHeLiuCor}, we have
\begin{align*}
g(A)&=\frac{m(m(b^n-1)-(b-1))}{2(b-1)}+\frac{mb^n(n-1)}{2}
+\left(\frac{m(b^n-1)}{2(b-1)}-\frac{1}{2}\right)(m(b^n-1)+d-1)
\\&=\frac{m(b^n-1)(mb^n+d-1)}{2(b-1)}+\frac{mb^n(n-2)+1-d}{2}.
\end{align*}
This completes the proof.
\end{proof}

In Theorem \ref{m2nd-repunit}, it is known that the embedding dimension $e(A)\leq n+1$. Setting $m=d=1$ in Theorem \ref{m2nd-repunit}, we have $b^{n}\cdot\frac{m(b^n-1)}{b-1}+\frac{b^{n}-1}{b-1}d=(b^n+1)a$. In fact, we have $e(A)=n$ when $m=d=1$ (see \cite{Rosales.Repunit}). But this does not affect that $\langle A\rangle$ becomes the repunit numerical semigroup \cite{Rosales.Repunit}. Clearly, we have the following corollary.

\begin{cor}{\em \cite[Theorem 20, Theorem 25]{Rosales.Repunit}}\label{correpunits}
Let $n,b\geq 2$, and set $m=d=1$ in Theorem \ref{m2nd-repunit}. Then $S(b,n)=\langle A\rangle$ is the repunit numerical semigroup. Furthermore, we have
\begin{align*}
&F(S(b,n))=\frac{b^n-1}{b-1}b^n-1,
\\&g(S(b,n))=\frac{b^n}{2}\left(\frac{b^n-1}{b-1}+n-2\right).
\end{align*}
\end{cor}

If $b=2$ in Corollary \ref{correpunits}, then $\langle A\rangle$ becomes the Mersenne numerical semigroup \cite{Rosales2016}.

\begin{cor}{\em \cite[Theorem 16, Theorem 19]{Rosales2016}}
Let $n\geq 2$, $b=2$ in Corollary \ref{correpunits}. Then $S(n)=\langle A\rangle$ is the Mersenne numerical semigroup. Furthermore, we have
\begin{align*}
&F(S(n))=2^{2n}-2^n-1,
\\&g(S(n))=2^{n-1}(2^n+n-3).
\end{align*}
\end{cor}

\begin{thm}\label{gener-repunit-type}
For $A=\left(a,ba+d,b^2a+\frac{b^2-1}{b-1}d,\ldots,b^ka+\frac{b^k-1}{b-1}d\right)$, assume $k=n$, $a=\frac{m(b^n-1)}{b-1}$, $b,n\geq 2$, $m\geq 1$, $d\in \mathbb{P}$, and $\gcd(a,d)=1$. Then $t(A)=n-1$. Furthermore,
$$PF(A)=\{F(A), F(A)-d,\ldots,F(A)-(n-2)d\}.$$
\end{thm}
\begin{proof}
Let $X(r)=(x_1,x_2,\ldots,x_{n})$ be a greedy presentation of $r$, $0\leq r\leq a-1$.
By Equation \eqref{Ndr-pre}, we have two representations of $N_{dr}$:
\begin{align}
N_{dr}&=\left(\sum_{i=1}^{n}b^ix_i\right)a+\sum_{i=1}^{n}\frac{b^i-1}{b-1}x_id\label{orderrelation}
\\&=\left(\sum_{i=1}^{n}b^ix_i\right)a+\sum_{i=1}^{n}\frac{b^ix_i}{b-1}d
-\sum_{i=1}^{n}\frac{x_i}{b-1}d.\label{orderrelation2}
\end{align}

By Equation \eqref{Ra-1repun}, the set $R(a-1)$ consists of all greedy presentations of $0\leq r\leq a-1$. We recall that the order relation $a\preceq_{\langle A\rangle} b$ if and only if $b-a \in \langle A\rangle$. In the Ap\'ery set
$Ape(A,a)=\{N_{d0},N_{d1},\ldots,N_{d(a-1)}\}$, by comparing the $x_i$'s in Equation \eqref{orderrelation}, we see that
if $N_{dr}$ is maximal (under $\preceq_{\langle A\rangle}$) in $Ape(A,a)$, then its corresponding $(x_1,\dots, x_{n})$ has to be one of the following greedy presentations:
$$(b,b-1,\ldots,b-1,m-1),\ (0,b,b-1,\ldots,b-1,m-1),\ (0,0,b,\ldots,b-1,m-1),\ \ldots ,\ (0,0,0,\ldots,b,m-1).$$

For the above candidates, we always have
$$\sum_{i=1}^{n}b^ix_i=mb^n.$$
For each candidate $X=(x_1,\ldots,x_{n})\in R(a-1)$ associated to $N_{dr}$, to show that $N_{dr}$ is maximal, it suffices to show that there is no candidate
$X^{\prime}=(x_1^{\prime},\ldots,x_{n}^{\prime})\in R(a-1)$ associated to $N_{dr'}$ such that $0\neq N_{dr'}-N_{dr} \in \langle A\rangle$.

Assume to the contrary the existence of the $X^\prime$. By Equation \eqref{orderrelation2}, it follows that
$$N_{dr'}-N_{dr}=\sum_{i=1}^{n}\frac{x_i}{b-1}d-\sum_{i=1}^{n}\frac{x_i'}{b-1}d=td,\ \ t=1,2,\ldots,n-2.$$
Then $td\in \langle A\rangle$, so that there exist not all zero nonnegative $(y_0,y_1,\ldots,y_{n})$ such that
$$td=y_0a+y_1(ba+d)+\cdots +y_{n}\left(b^{n}a+\frac{b^{n}-1}{b-1}d\right).$$
Then we have
$$(n-2)d\geq \left(t-y_1-y_2\frac{b^2-1}{b-1}-\cdots -y_{n}\frac{b^{n}-1}{b-1}\right)d=(y_0+by_1+\cdots +b^{n}y_{n})a>0.$$
By $\gcd(a,d)=1$, $a$ has to divide $t-y_1-y_2\frac{b^2-1}{b-1}-\cdots -y_{n}\frac{b^{n}-1}{b-1}\leq t \leq n-2<a$. Thus
 $(y_0+by_1+\cdots +b^{n}y_{n})a=0$, a contradiction.

Therefore,
$$\max\nolimits_{\preceq_{\langle A\rangle}} Ape(A, a)=\{N_{dr} \mid r\in a-\{1,2,\ldots,n-1\}\}.$$
By Lemma \ref{Pseudo-FP-Prop}, we have
\begin{align*}
PF(A)&=\{ N_{dr}-a \mid r\in a-\{1,2,\ldots,n-1\}\}
\\&=\{F(A),F(A)-d,\ldots,F(A)-(n-2)d\}.
\end{align*}
This completes the proof.
\end{proof}

Similarly, we can obtain the following corollary.
\begin{cor}{\em \cite[Theorem 23]{Rosales.Repunit}}\label{Fsotyperep}
Let $n,b\geq 2$, and set $m=d=1$ in Theorem \ref{m2nd-repunit}. Then $S(b,n)=\langle A\rangle$ is the repunit numerical semigroup. We have $t(S(b,n))=n-1$ and
\begin{align*}
PF(S(b,n))=\{F(S(b,n)),F(S(b,n))-1,\ldots,F(S(b,n))-(n-2)\}.
\end{align*}
\end{cor}

\begin{cor}{\em \cite[Theorem 18]{Rosales2016}}
Let $n\geq 2$, $b=2$ in Corollary \ref{Fsotyperep}. Then $S(n)=\langle A\rangle$ is the Mersenne numerical semigroup. We have $t(S(n))=n-1$ and
\begin{align*}
PF(S(n))=\{F(S(n)),F(S(n))-1,\ldots,F(S(n))-(n-2)\}.
\end{align*}
\end{cor}

\section{A Generalization of Thabit Numerical Semigroups}

In this section, we focus on Collection \texttt{GCNS} in the case $b_i=2^i-1$, $u=1$, $k=n+1$, $a=3\cdot 2^n-1$, $n\geq 1$, and $d\in \mathbb{P}$. That is, we consider the
numerical semigroup generated by
$$A=\big( 3\cdot 2^n-1,3\cdot 2^{n+1}-2+d, 3\cdot 2^{n+2}-4+3d,\ldots, 3\cdot 2^{2n+1}+2^{n+1}(d-1)-d \big).$$
This generalizes the Thabit numerical semigroup. We solve the Frobenius problem by giving closed formulas for the Frobenius number, genus, and pseudo-Frobenius numbers.

By Proposition \ref{greedproper}, $a-1=3\cdot 2^n-2=(2^{n+1}-1)+(2^n-1)$, and $2^{n+1}-2=2(2^n-1)$, we have
\begin{align}\label{Ra-1construct}
R(a-1)&=\Big\{(x_1,x_2,\ldots,x_{n+1}) \mid (x_n,x_{n+1})=(0,0),(1,0),(0,1);\notag
\\&\ \ \ 0\leq x_i\leq 2\ \ \text{for}\ \ 1\leq i\leq n-1;\ \
\text{if}\ \ x_i=2,\ \ \text{then}\ \  x_j=0\ \  \text{for}\ \ j\leq i-1\Big\}\notag
\\&\ \ \ \biguplus \Big\{(0,0,\ldots,0,x_{n},x_{n+1})\mid (x_n,x_{n+1})=(2,0),(1,1)\Big\}.
\end{align}
It can be verified that $\#R(a-1)=3\cdot 2^n-1=a$.

\begin{thm}\label{Thabit-32nd}
Let $A=\big( 3\cdot 2^n-1,3\cdot 2^{n+1}-2+d, 3\cdot 2^{n+2}-4+3d,\ldots, 3\cdot 2^{2n+1}-2^{n+1}+(2^{n+1}-1)d \big)$, where $n,d\in \mathbb{P}$, $n\geq 1$, and $\gcd(3\cdot 2^n-1,d)=1$. We have
\begin{align*}
&F(A)=9\cdot 2^{2n}+3(d-2)\cdot 2^n-2d+1,
\\&g(A)=9\cdot 2^{2n-1}+(3n-8)\cdot 2^{n-1}+(3\cdot 2^{n-1}-1)d+1.
\end{align*}
\end{thm}
\begin{proof}
By Corollary \ref{HaHeLiuCor} along with $a-1=(2^{n+1}-1)+(2^n-1)$, we obtain
\begin{align*}
F(A)&=(a+d-2+2)a-d=9\cdot 2^{2n}+3(d-2)\cdot 2^n-2d+1.
\end{align*}

Now we consider the genus $g(A)$. Based on the composition of $R(a-1)$ in Equation \eqref{Ra-1construct} and Corollary \ref{HaHeLiuCor}, we have
\begin{align*}
g(A)&=\sum_{r=1}^{a-1}\left(\sum_{i=1}^{n+1}x_i\right)_r+\frac{(a-1)(a+d-1)}{2}
\\&=4+3\left(\sum_{i=0}^{n-1}i\binom{n-1}{i}
+\sum_{j=0}^{n-1}\left(\sum_{i=1}^{n-j-1}(i+2)\binom{n-j-1}{i}\right)\right)
\\&\ \ +2\left(\sum_{i=1}^{n-1}\binom{n-1}{i}+\sum_{j=1}^{n-1}\sum_{i=0}^{n-j-1}\binom{n-j-1}{i}+1\right)
+\frac{(a-1)(a+d-1)}{2}
\\&=6+3\left(2^{n-2}(n-1)+2^{n-2}(n-3)+1+2^n-2\right)
\\&\ \ +2\left(2^{n-1}-1+\sum_{j=1}^{n-1}2^{n-j-1}\right)
+\frac{(a-1)(a+d-1)}{2}
\\&=6+3(2^{n-1}\cdot n-1)+2^n-2+2^n-2+(3\cdot 2^{n-1}-1)(3\cdot 2^n-2+d)
\\&=9\cdot 2^{2n-1}+(3n-8)\cdot 2^{n-1}+(3\cdot 2^{n-1}-1)d+1.
\end{align*}
This completes the proof.
\end{proof}

Setting $d=1$ in Theorem \ref{Thabit-32nd} gives the following corollary.

\begin{cor}{\em \cite[Corollary 20, Theorem 29]{Rosales2015}}
Let $d=1$ in Theorem \ref{Thabit-32nd}, such that $T(n)=\langle A\rangle$ is the Thabit numerical semigroup in \cite{Rosales2015}. Then we have
\begin{align*}
&F(T(n))=9\cdot 2^{2n}-3\cdot 2^n-1,
\\&g(T(n))=9\cdot 2^{2n-1}+(3n-5)2^{n-1}.
\end{align*}
\end{cor}

Let $X(r)=(x_1,x_2,\ldots,x_{n+1})$ be a greedy presentation of $r$, $0\leq r\leq a-1$.
According to Equation \eqref{Ndr-pre}, we obtain
\begin{align}\label{equaNdrT}
N_{dr}&=\left(\sum_{i=1}^{n+1}2^ix_i\right)a+\sum_{i=1}^{n+1}(2^i-1)x_id=\left(\sum_{i=1}^{n+1}2^ix_i\right)(a+d)-\sum_{i=1}^{n+1}x_id.
\end{align}

Under the order relation $\preceq_{\langle A\rangle}$ and Equation \eqref{Ra-1construct}, we know that the greedy presentations of the maximal elements in $Ape(A,a)$ are contained in the following elements:
\begin{align*}
&(2,1,1,\ldots,1,1,0),\ (0,2,1,\ldots,1,1,0),\ (0,0,2,\ldots,1,1,0),\ \ldots,\ (0,0,0,\ldots,2,1,0)
\\&(2,1,1,\ldots,1,0,1),\ (0,2,1,\ldots,1,0,1),\ (0,0,2,\ldots,1,0,1),\ \ldots,\ (0,0,0,\ldots,2,0,1)
\\&(0,0,0,\ldots,0,2,0),\ (0,0,0,\ldots,0,1,1).
\end{align*}
For the greedy presentations of the first two lines above, we always have $\sum_{i=1}^{n-1}2^ix_i=2^n$.
By Equation \eqref{equaNdrT}, we get
\begin{align*}
&(2^n+2^{n+1})(a+d)-(t+1)d-((2^n+2^n)(a+d)-td)
\\=&2^n(a+d)-d=2^na+(2^n-1)d\in \langle A\rangle,
\end{align*}
where $2\leq t\leq n$. Therefore, we have
\begin{align*}
\max\nolimits_{\preceq_{\langle A\rangle}} Ape(A, a)\subseteq\Big\{(2^n+2^{n+1})(a+d)-\{2d,3d,\ldots,(n+1)d\}\Big\}\biguplus \{2^{n+1}(a+d)-(n+1)d\}.
\end{align*}

Similar to the proof of Theorem \ref{gener-repunit-type}, we  obtain $s\cdot d\notin \langle A\rangle$, $1\leq s\leq (n-1)$. Furthermore, we have
$$\Big\{(2^n+2^{n+1})(a+d)-\{2d,3d,\ldots,(n+1)d\}\Big\}\subseteq \max\nolimits_{\preceq_{\langle A\rangle}} Ape(A,a).$$

\begin{thm}\label{maxmial-Ape-Thab}
Let $A=\big( 3\cdot 2^n-1,3\cdot 2^{n+1}-2+d, 3\cdot 2^{n+2}-4+3d,\ldots, 3\cdot 2^{2n+1}-2^{n+1}+(2^{n+1}-1)d \big)$, where $n,d\in \mathbb{P}$, $n\geq 2$, and $\gcd(3\cdot 2^n-1,d)=1$. Then
\begin{align*}
\max\nolimits_{\preceq_{\langle A\rangle}} Ape(A, a)=\Big\{(2^n+2^{n+1})(a+d)-\{2d,3d,\ldots,(n+1)d\}\Big\}\biguplus \{2^{n+1}(a+d)-(n+1)d\},
\end{align*}
where $a=3\cdot 2^n-1$.
\end{thm}
\begin{proof}
It suffices to show that $\{2^{n+1}(a+d)-(n+1)d\}\subseteq \max\nolimits_{\preceq_{\langle A\rangle}} Ape(A, a)$.
This is similar to the proof of Theorem 24 in \cite{Rosales2015}. We omit it.
\end{proof}

By applying Theorem \ref{maxmial-Ape-Thab} and Lemma \ref{Pseudo-FP-Prop} we get the following result.
\begin{thm}\label{Pseudo-Thab-Haha}
Let $A=\big(3\cdot 2^n-1,3\cdot 2^{n+1}-2+d, 3\cdot 2^{n+2}-4+3d,\ldots, 3\cdot 2^{2n+1}-2^{n+1}+(2^{n+1}-1)d \big)$, where $n,d\in \mathbb{P}$, $n\geq 2$, and $\gcd(3\cdot 2^n-1,d)=1$. Then $t(A)=n+1$. Furthermore
$$PF(A)=\{F(A),F(A)-d,\ldots,F(A)-(n-1)d\}\biguplus \{6\cdot 2^{2n}+(2d-5)2^n-(n+1)d+1\}.$$
\end{thm}

\begin{cor}{\em \cite[Corollary 25]{Rosales2015}}
Let $d=1$ in Theorem \ref{Pseudo-Thab-Haha}. Then $T(n)=\langle A\rangle$ is the Thabit numerical semigroup in \cite{Rosales2015}. We have $t(T(n))=n+1$ and
$$PF(A)=\{F(A),F(A)-1,\ldots,F(A)-(n-1)\}\biguplus \{6\cdot 2^{2n}-3\cdot 2^n-n\}.$$
\end{cor}

\section{Proth Numerical Semigroups}\label{SecPorthNS}

Now, we consider the Proth numerical semigroup $P_m(n)=\{m2^{n+i}+1\mid i\in\mathbb{N}\}$, where $m$ is an odd positive number and $m<2^n$. So we assume that $2^r<m<2^{r+1}$ for some $r\in \mathbb{N}$. In \cite{P.Srivastava}, Srivastava and Thakkar obtain the embedding dimension $e(P_m(n))=n+r+1$. Therefore, we have $P_m(n)=\langle A\rangle$, where
\begin{align*}
A&=(m2^n+1,m2^{n+1}+1,m2^{n+2}+1,\ldots,m2^{2n+r}+1)
\\&=(m2^n+1,2(m2^n+1)-1,2^2(m2^n+1)-3,\ldots,2^{n+r}(m2^n+1)-(2^{n+r}-1)).
\end{align*}

Srivastava and Thakkar derive the Frobenius number for the case $m=2^r+1$. Now we use Theorem \ref{HBuadNDr} to calculate $F(P_{2^r+1}(n))$. Note that we cannot directly use the result of Corollary \ref{HaHeLiuCor} to calculate the Frobenius number.
\begin{prop}{\em \cite[Theorem 6]{P.Srivastava}}\label{PropPSrivastava}
Let $n>2$. We obtain
$$F(P_{2^r+1}(n))=4^{n+r}+4^n+2^{2n+r+1}+3\cdot 2^{n+r}+3\cdot 2^n+3.$$
\end{prop}
\begin{proof}
In Collection \texttt{CNS}, we set $a=(2^r+1)2^n+1$, $b=2$, $d=-1$, and $k=n+r$. This is the numerical semigroup $P_{2^r+1}(n)$. Now we consider the greedy presentation of $a-1=(2^r+1)2^n$.
By Proposition \ref{greedproper} and $b_i=2^i-1$ for $1\leq i\leq n+r$, it follows that
$$x_{n+r}=\left\lfloor \frac{2^{n+r}+2^n}{2^{n+r}-1}\right\rfloor=1,\ \ x_n=\left\lfloor \frac{a-1-x_{n+r}b_{n+r}}{2^n-1}\right\rfloor=1,\ \ \text{and}\ \ x_1=2.$$
Thus, we have
$a-1=x_1\cdot 1+x_n(2^n-1)+x_{n+r}(2^{n+r}-1)$. By Theorem \ref{HBuadNDr}, we obtain
$$N_{d(a-1)}-N_{d(a-2)}=(w((2^r+1)2^n)-w((2^r+1)2^n-1))a-1=2a-1.$$
By Lemma \ref{colex-incre}, we get $w(a-1)\geq w(a-2)\geq \cdots \geq w(1)$. Hence, we have
$N_{d(i+1)}-N_{di}\geq -1$ for $1\leq i\leq a-2$. We can easily obtain $N_{d(a-1)}\geq N_{dj}$ for any $1\leq j\leq a-2$. By Lemma \ref{LiuXin001} and Theorem \ref{HBuadNDr}, we conclude that
$$F(P_{2^r+1}(n))=N_{d(a-1)}-a=4a+(a-1)(a-1)-a=4^{n+r}+4^n+2^{2n+r+1}+3\cdot 2^{n+r}+3\cdot 2^n+3.$$
This completes the proof.
\end{proof}

Srivastava and Thakkar also obtain the pseudo-Frobenius numbers $PF(P_{2^r+1}(n))$, the type $t(P_{2^r+1}(n))$, and an inequality of the genus $g(P_{2^r+1}(n))$ in \cite{P.Srivastava}. Note that they did not obtain a formula for the genus. Now we consider the Open Problem \ref{OpenQuestProthFP}. We first give a characterization of Ap\'ery set for $P_m(n)$.

\begin{thm}
Let $2^r<m<2^{r+1}$ for some $r\in \mathbb{N}$, where $m$ is an odd positive number, and $m<2^n$. Let $n>2$ be a positive integer. Then the Ap\'ery set of $P_m(n)$ is
\begin{align*}
\{N_{-j}\mid 0\leq j\leq m2^n\}=\{0,N_{-1},N_{-2},\ldots,N_{-m2^n}\}.
\end{align*}
If $X(j)=(x_1,x_2,\ldots,x_{n+r})$ is the greedy presentation of $j$ in Proposition \ref{greedproper}, then
$$N_{-j}=\left(\sum_{i=1}^{n+r}x_i \right)(m2^n+1)+jm2^n=w(j)\cdot (m2^n+1)-j.$$
where $w(j)=\sum_{i=1}^k2^ix_i$ is the weight of $j$.
\end{thm}
\begin{proof}
Srivastava and Thakkar obtain the embedding dimension $e(P_m(n))=n+r+1$ in \cite{P.Srivastava}. Therefore, we can set $k=n+r$ in Theorem \ref{HBuadNDr}. Furthermore, we set $a=m2^n+1$, $d=-1$, $h_i=2^i$, and $b_i=2^i-1$ for $1\leq i\leq n+r$ in Theorem \ref{HBuadNDr}. Thus, we get $u=1$, $s_1=s_2=\cdots =s_{k-1}=2$. The restriction in Theorem \ref{HBuadNDr} holds. We have
$$N_{-j}=\left(\sum_{i=1}^{n+r}x_i\right)a+j(ua+d)=\left(\sum_{i=1}^{n+r}x_i\right)(m2^n+1)+jm2^n.$$
Similarly, by Theorem \ref{HBuadNDr}, we also obtain
$N_{-j}=w(j)\cdot (m2^n+1)-j$.
This completes the proof.
\end{proof}

We can also obtain the formulas for the genus of the Proth numerical semigroup $P_m(n)$.
\begin{thm}
Let $2^r<m<2^{r+1}$ for some $r\in \mathbb{N}$, where $m$ is an odd positive number and $m<2^n$. Let $n>2$ be a positive integer. Then we have
\begin{align*}
g(P_m(n))=\sum_{j=1}^{m2^n}\left(\sum_{i=1}^{n+r}x_i\right)_{j}+\frac{m2^n(m2^n-1)}{2},
\end{align*}
where $\left(\sum_{i=1}^{n+r}x_i\right)_j$ denotes the sum of elements in the greedy presentation of $j$.
\end{thm}
\begin{proof}
We set $k=n+r$, $a=m2^n+1$, $b=2$, and $d=-1$ in Corollary \ref{HaHeLiuCor111111}. We can easily verify the condition $a>k-1-\frac{d-1}{b-1}$ holds. Therefore, we have
\begin{align*}
g(P_m(n))=\sum_{j=1}^{m2^n}\left(\sum_{i=1}^{n+r}x_i\right)_{j}+\frac{m2^n(m2^n-1)}{2}.
\end{align*}
This completes the proof.
\end{proof}

We now consider the Frobenius number of $P_m(n)$.
\begin{thm}\label{PorthNSX12}
Let $q\in \mathbb{P}$ be an odd number and $r\in \mathbb{N}$. Let $m=2^r+q<2^{r+1}$ and $n\geq r+1$. If the greedy presentation of $(2^r+q)2^n$ is $(x_1,x_2,\ldots,x_{n+r})$ and $x_1\neq 0$, then we have $$F(P_{2^r+q}(n))=\left(\sum_{i=1}^{n+r}x_i\right)((2^r+q)2^n+1)+((2^r+q)2^n)^2-(2^r+q)2^n-1.$$
\end{thm}
\begin{proof}
Since $q$ is an odd positive integer, $m$ is an odd number.
Given that $2^r<m=2^r+q<2^{r+1}$ and $n\geq r+1$, it follows that $m<2^n$. Therefore, $P_{2^r+q}(n)$ is a Proth numerical semigroup.
In Collection \texttt{CNS}, we set $a=(2^r+q)2^n+1$, $b=2$, $d=-1$, and $k=n+r$.
Because the greedy presentation of $a-1$ is $(x_1,x_2,\ldots,x_{n+r})$ and $x_1\neq 0$, we know $x_1\in\{1,2\}$ and the greedy presentation of $a-2$ is $(x_1-1,x_2,\ldots,x_{n+r})$.
By Theorem \ref{HBuadNDr} and the weight of $a-1$ is $w(a-1)=\sum_{i=1}^k(2^i-1+1)x_i$, we obtain
$$N_{d(a-1)}-N_{d(a-2)}=(w(a-1)-w(a-2))a-1=(2x_1-2(x_1-1))a-1=2a-1.$$
According to Lemma \ref{colex-incre}, we have $w(a-1)\geq w(a-2)\geq \cdots \geq w(1)$.
Therefore, by Theorem \ref{HBuadNDr}, it follows that $N_{d(i+1)}-N_{di}\geq -1$ for $1\leq i\leq a-2$.
Furthermore, we get
$$N_{d(a-1)}-N_{d(a-3)}\geq 2a-2,\ \ \ N_{d(a-1)}-N_{d(a-4)}\geq 2a-3,\ldots,\ \ \ N_{d(a-1)}-N_{d(1)}\geq a+2.$$
By Lemma \ref{LiuXin001} and Theorem \ref{HBuadNDr}, we conclude that
$$F(P_{2^r+q}(n))=N_{d(a-1)}-a=\sum_{i=1}^{n+r}x_ia+(a-1)(a-1)-a.$$
This completes the proof.
\end{proof}

In Theorem \ref{PorthNSX12}, we can obtain Proposition \ref{PropPSrivastava} by setting $q=1$.

\begin{thm}\label{PorthNSX001}
Let $q\in \mathbb{P}$ be an odd number and $r\in \mathbb{N}$. Let $m=2^r+q<2^{r+1}$ and $n\geq r+1$. If the greedy presentation of $(2^r+q)2^n$ is $(x_1,x_2,\ldots,x_{n+r})$ and $x_1=0$, $x_2\neq 0$, then we have
$$F(P_{2^r+q}(n))=\left(\sum_{i=1}^{n+r}x_i\right)((2^r+q)2^n+1)+((2^r+q)2^n-1)(2^r+q)2^n.$$
\end{thm}
\begin{proof}
Similar to the proof of Theorem \ref{PorthNSX12}, we set $a=(2^r+q)2^n+1$, $b=2$, $d=-1$, and $k=n+r$ in Collection \texttt{GCNS}. Because the greedy presentation of $a-1$ is $(0,x_2,\ldots,x_{n+r})$ and $x_2\in\{1,2\}$, the greedy presentation of $a-2$ is $(2,x_2-1,\ldots,x_{n+r})$, and the greedy presentation of $a-3$ is $(1,x_2-1,\ldots,x_{n+r})$.
Thus, we get
$$w(a-2)-w(a-1)=(3+1)(x_2-1)+(1+1)2-(3+1)x_2=0,\ \ w(a-2)-w(a-3)=2.$$
By Theorem \ref{HBuadNDr}, we obtain
$$N_{d(a-2)}-N_{d(a-1)}=(w(a-1)-w(a-2))a-d=1,\ \ N_{d(a-2)}-N_{d(a-3)}=2a-1.$$
by Theorem \ref{HBuadNDr}, it follows that
$N_{d(i+1)}-N_{di}\geq -1$ for $1\leq i\leq a-4$.
Therefore, we have
$$N_{d(a-2)}-N_{d(a-4)}\geq 2a-2,\ \ \ N_{d(a-2)}-N_{d(a-5)}\geq 2a-3,\ldots,\ \ \ N_{d(a-2)}-N_{d(1)}\geq a+3.$$
By Lemma \ref{LiuXin001} and Theorem \ref{HBuadNDr}, we obtain
$$F(P_{2^r+q}(n))=N_{d(a-2)}-a=\left(1+\sum_{i=1}^{n+r}x_i\right)a+(a-2)(a-1)-a.$$
This completes the proof.
\end{proof}

Theorem \ref{PorthNSX12} and Theorem \ref{PorthNSX001} do not cover all cases for an odd number $m$ (or an odd number $q$). For example, if $q=63$, then we have $x_1=x_2=0$ and $x_3=1$ in the greedy presentation of $(2^r+q)2^n$. However, our idea can be applied to a concrete $m$. For $q<50$, we obtain Table \ref{PorthFor50}, which provides the explicit formula. In Table \ref{PorthFor50}, the elements in $**$ are all 0.

\begin{exa}
If $q=3$, $r\geq 2$, and $n\geq 3$, then the sum of elements in the greedy presentation of $(2^r+q)2^n$ is $\sum_{i=1}^{n+r}x_i=4$ by Table \ref{PorthFor50}. By Theorem \ref{PorthNSX001}, we obtain
$$F(P_{2^r+3}(n))=4((2^r+3)2^n+1)+((2^r+3)2^n-1)(2^r+3)2^n.$$
Similarly, if $q=7$, $r\geq 3$ and $n\geq 4$, then we have
$$F(P_{2^r+7}(n))=6((2^r+7)2^n+1)+((2^r+7)2^n)^2-(2^r+7)2^n-1$$
by Theorem \ref{PorthNSX12}.
\end{exa}

\begin{tiny}
\begin{table}[htbp]
    	\centering
    	\caption{The greedy presentation of $(2^r+q)2^n$ for $q<50$.}\label{PorthFor50}
    	\begin{tabular}{c|c|c|c|c}
    		\hline \hline
Value of $q$ & Value of $r$ & Value of $n$ & $(x_1,x_2,x_3,**,x_n,x_{n+1},x_{n+2},x_{n+3},x_{n+4},x_{n+5},**,x_{n+r})$ for $a-1$ & Theorem    \\
    		\hline
$1$ & $r\geq 0$ & $n\geq 1$ & $(2,0,0,**,1,0,0,0,0,0,**,1)$ & \text{Theorem \ref{PorthNSX12}}   \\
    		\hline
$3$ & $r\geq 2$ & $n\geq 3$ & $(0,1,0,**,1,1,0,0,0,0,**,1)$ & \textcolor{blue}{Theorem \ref{PorthNSX001}}   \\
    		\hline
$5$ & $r\geq 3$ & $n\geq 4$ & $(0,1,0,**,1,0,1,0,0,0,**,1)$ & \textcolor{blue}{Theorem \ref{PorthNSX001}}   \\
    		\hline
$7$ & $r\geq 3$ & $n\geq 4$ & $(1,1,0,**,1,1,1,0,0,0,**,1)$ & \text{Theorem \ref{PorthNSX12}}   \\
    		\hline
$9$ & $r\geq 4$ & $n\geq 5$ & $(0,1,0,**,1,0,0,1,0,0,**,1)$ & \textcolor{blue}{Theorem \ref{PorthNSX001}}   \\
    		\hline
$11$ & $r\geq 4$ & $n\geq 5$ & $(1,1,0,**,1,1,0,1,0,0,**,1)$ & \text{Theorem \ref{PorthNSX12}}   \\
    		\hline
$13$ & $r\geq 4$ & $n\geq 5$ & $(1,1,0,**,1,0,1,1,0,0,**,1)$ & \text{Theorem \ref{PorthNSX12}}   \\
    		\hline
$15$ & $r\geq 4$ & $n\geq 5$ & $(2,1,0,**,1,1,1,1,0,0,**,1)$ & \text{Theorem \ref{PorthNSX12}}   \\
    		\hline
$17$ & $r\geq 5$ & $n\geq 6$ & $(0,1,0,**,1,0,0,0,1,0,**,1)$ & \textcolor{blue}{Theorem \ref{PorthNSX001}}   \\
    		\hline
$19$ & $r\geq 5$ & $n\geq 6$ & $(1,1,0,**,1,1,0,0,1,0,**,1)$ & \text{Theorem \ref{PorthNSX12}}   \\
    		\hline
$21$ & $r\geq 5$ & $n\geq 6$ & $(1,1,0,**,1,0,1,0,1,0,**,1)$ & \text{Theorem \ref{PorthNSX12}}   \\
    		\hline
$23$ & $r\geq 5$ & $n\geq 6$ & $(2,1,0,**,1,1,1,0,1,0,**,1)$ & \text{Theorem \ref{PorthNSX12}}   \\
    		\hline
$25$ & $r\geq 5$ & $n\geq 6$ & $(1,1,0,**,1,0,0,1,1,0,**,1)$ & \text{Theorem \ref{PorthNSX12}}   \\
    		\hline
$27$ & $r\geq 5$ & $n\geq 6$ & $(2,1,0,**,1,1,0,1,1,0,**,1)$ & \text{Theorem \ref{PorthNSX12}}   \\
    		\hline
$29$ & $r\geq 5$ & $n\geq 6$ & $(2,1,0,**,1,0,1,1,1,0,**,1)$ & \text{Theorem \ref{PorthNSX12}}   \\
    		\hline
$31$ & $r\geq 5$ & $n\geq 6$ & $(0,2,0,**,1,1,1,1,1,0,**,1)$ & \textcolor{blue}{Theorem \ref{PorthNSX001}}   \\
    		\hline
$33$ & $r\geq 6$ & $n\geq 7$ & $(0,1,0,**,1,0,0,0,0,1,**,1)$ & \textcolor{blue}{Theorem \ref{PorthNSX001}}   \\
    		\hline
$35$ & $r\geq 6$ & $n\geq 7$ & $(1,1,0,**,1,1,0,0,0,1,**,1)$ & \text{Theorem \ref{PorthNSX12}}   \\
    		\hline
$37$ & $r\geq 6$ & $n\geq 7$ & $(1,1,0,**,1,0,1,0,0,1,**,1)$ & \text{Theorem \ref{PorthNSX12}}   \\
    		\hline
$39$ & $r\geq 6$ & $n\geq 7$ & $(2,1,0,**,1,1,1,0,0,1,**,1)$ & \text{Theorem \ref{PorthNSX12}}   \\
    		\hline
$41$ & $r\geq 6$ & $n\geq 7$ & $(1,1,0,**,1,0,0,1,0,1,**,1)$ & \text{Theorem \ref{PorthNSX12}}   \\
    		\hline
$43$ & $r\geq 6$ & $n\geq 7$ & $(2,1,0,**,1,1,0,1,0,1,**,1)$ & \text{Theorem \ref{PorthNSX12}}   \\
    		\hline
$45$ & $r\geq 6$ & $n\geq 7$ & $(2,1,0,**,1,0,1,1,0,1,**,1)$ & \text{Theorem \ref{PorthNSX12}}   \\
    		\hline
$47$ & $r\geq 6$ & $n\geq 7$ & $(0,2,0,**,1,1,1,1,0,1,**,1)$ & \textcolor{blue}{Theorem \ref{PorthNSX001}}   \\
    		\hline
$49$ & $r\geq 6$ & $n\geq 7$ & $(1,1,0,**,1,0,0,0,1,1,**,1)$ & \text{Theorem \ref{PorthNSX12}}   \\
    		\hline
    	\end{tabular}
    \end{table}
\end{tiny}

\section{Other Numerical Semigroups}

Our collection $A=\left(a,ba+d,b^2a+\frac{b^2-1}{b-1}d,\ldots,b^ka+\frac{b^k-1}{b-1}d\right)$ in Corollary \ref{HaHeLiuCor} can also be used to explain the following six numerical semigroups. The first five are considered in the literature. The sixth item seems new, and we will give a proof.

1. If $a=(2^m-1)\cdot 2^n-1$, $b=2$, $d=1$, $k=n+m-1$, $n\geq 1$, and $2\leq m\leq 2^n$, then
$$A=((2^m-1)\cdot 2^n-1, (2^m-1)\cdot 2^{n+1}-1,\ldots,(2^{m}-1)\cdot 2^{2n+m-1}-1),$$
and $\langle A\rangle$ is a class of numerical semigroups $S(m,n)$ in \cite{GuZeTang}.

2. Let $m,n\in \mathbb{N}$, $m\geq 2$, and
$$\begin{aligned}
  \delta=
\left\{
    \begin{array}{lc}
    1&\ \text{if}\ \ n=0;\ \ \ \ \ \ \ \ \ \ \ \\
    m&\ \text{if}\ \ n\neq 0, m\leq n;\\
    m-1&\ \text{if}\ \ n\neq 0, m> n.\\
    \end{array}
\right.
\end{aligned}$$
If $a=(2^m+1)\cdot 2^n-(2^m-1)$, $b=2$, $d=2^m-1$, and $k=n+\delta$, then $\langle A\rangle$ is a class of numerical semigroups $GT(n,m)$ in \cite{KyunghwanSong}.

3. If $a=b^{n+1}+\frac{b^n-1}{b-1}$, $b\geq 2$, $d=1$, and $k=n+1$, then
$$A=\left(b^{n+1}+\frac{b^n-1}{b-1}, b^{n+2}+\frac{b^{n+1}-1}{b-1},\ldots,b^{2n+2}+\frac{b^{2n+1}-1}{b-1}\right),$$
and $\langle A\rangle$ is a class of numerical semigroups $S(b,n)$ in \cite{GuZe2020}.

4. If $a=(b+1)b^n-1$, $b\geq 2$, $d=b-1$, and $k=n+1$, then
$$A=((b+1)b^n-1,(b+1)b^{n+1}-1,\ldots,(b+1)b^{2n+1}-1),$$
and $\langle A\rangle$ is the Thabit numerical semigroup $T_{b,1}(n)$ of the first kind base $b$ in \cite{KyunghwanSong2020}.

5. If $a\geq 2$, $d\in \mathbb{P}$, and $k=\min\{m-1\mid b^m-1>(b-1)(a-1)\}$, then $\langle A\rangle$ is the numerical semigroup $G_{b,d}(a)$ in \cite{S.Ugolini}. This semigroup is derived by defining an affine map. Ugolini obtains the Frobenius number $F(G_{b,d}(a))$ and genus $g(G_{b,d}(a))$ by first calculating $e(G_{b,d}(a))$. Therefore, for the study of numerical semigroups $G_{b,d}(a)$, there are restrictions $a\geq 2$ and $k=\min\{m-1\mid b^m-1>(b-1)(a-1)\}$ in \cite{S.Ugolini}.
When studying the problem using Corollary \ref{HaHeLiuCor}, we obtain different restrictions:  $a\geq 2$, $k\geq 2$, and $a\geq k-1-\frac{d-1}{b-1}$.

6. If $b=2$, $a=m(2^k-1)+2^{k-1}-1$, $m\geq 1$, $d\in \mathbb{P}$, and $k\geq 3$, then $a-1=m(2^k-1)+2(2^{k-2}-1)$. We solve this case as follows. We solve the Frobenius problem by providing closed formulas for the Frobenius number and genus.
By Corollary \ref{HaHeLiuCor}, we have
\begin{align*}
F(A)&=\left(m(2^k-1)+2^{k-1}-1+d-2+m+2\right)\cdot (m(2^k-1)+2^{k-1}-1)-d
\\&=\left((2m+1)2^{k-1}-1\right)^2+(d-m)\left((2m+1)2^{k-1}-1\right)-md-d.
\end{align*}
Now we consider the genus $g(A)$. Let
\begin{align*}
R_1&=\Big\{(x_1,x_2,\ldots,x_{k}) \mid 0\leq x_k\leq m-1;\ \ \ 0\leq x_i\leq 2\ \ \text{for}\ \ 1\leq i\leq k-1,
\\&\ \ \ \ \ \ \ \ \ \ \text{if}\ \ x_i=2,\ \ \text{then}\ \  x_j=0\ \  \text{for}\ \ j< i\leq k-1\Big\}
\end{align*}
and
\begin{align*}
R_2&=\Big\{(x_1,x_2,\ldots,x_{k}) \mid x_k= m;\ x_{k-1}=0;\ \ 0\leq x_i\leq 2\ \text{for}\ \ 1\leq i\leq k-2,
\\&\ \ \ \ \ \ \ \ \ \ \text{if}\ \ x_i=2,\ \ \text{then}\ \  x_j=0\ \  \text{for}\ \ j< i\leq k-2\Big\}.
\end{align*}
Readers can easily verify that $R(a-1)=R_1\biguplus R_2$.
Similar to the proof of Theorem \ref{m2nd-repunit}, we have
\begin{align*}
g(A)&=\sum_{r=1}^{a-1}\left(\sum_{i=1}^{k}x_i\right)_r+\frac{(a-1)(a+d-1)}{2}
\\&=\sum_{i=1}^{m-1}i(2^k-1)+m(2^{k-1}-1)+\sum_{r=1}^{a-1}\left(\sum_{i=1}^{k-1}x_i\right)_r
+\frac{(a-1)(a+d-1)}{2}.
\end{align*}
Now, we obtain
\begin{align*}
\sum_{r=1}^{a-1}\left(\sum_{i=1}^{k-1} x_i\right)_r
=&m\left(\sum_{i=0}^{k-1}i\binom{k-1}{i}
+\sum_{j=1}^{k-1}\left(\sum_{i=0}^{k-j-1}(i+2)\binom{k-j-1}{i}\right)\right)
\\&\ \ +\left(\sum_{i=0}^{k-2}i\binom{k-2}{i}
+\sum_{j=1}^{k-2}\left(\sum_{i=0}^{k-j-2}(i+2)\binom{k-j-2}{i}\right)\right)
\\=&m(2^{k-1}k-1)+2^{k-2}(k-1)-1.
\end{align*}
Therefore, we get
\begin{align*}
g(A)=&2^{k-1}(2^k-1)m^2+\frac{1}{2}(d-1)(2^k-1)m+\left(2^{2k-1}+k2^{k-1}-2^{k+1}\right)m
\\&\ \ +2^{2k-3}+(d+k)2^{k-2}-5\cdot 2^{k-2}-d+1.
\end{align*}

The above discussion is summarized as follows.
\begin{thm}
Let $A=(a, 2a+d, 2^2a+3d, \ldots, 2^ka+(2^k-1)d)$, $a=m(2^k-1)+2^{k-1}-1$, $m\geq 1$, $k\geq 3$, $d\in \mathbb{P}$, and $\gcd(a,d)=1$. Then we have
\begin{align*}
F(A)=&\left((2m+1)2^{k-1}-1\right)^2+(d-m)\left((2m+1)2^{k-1}-1\right)-md-d,
\\g(A)=&2^{k-1}(2^k-1)m^2+\frac{1}{2}(d-1)(2^k-1)m+(2^{2k-1}+k2^{k-1}-2^{k+1})m
\\&\ \ +2^{2k-3}+(d+k)2^{k-2}-5\cdot 2^{k-2}-d+1.
\end{align*}
\end{thm}

\section{Concluding Remark}
This paper combines the contents of the previous version and a subsequent draft \cite{LiuXinYeYin}.
In the previous version, we mainly dealt with the case $A=(a, 2a+d, 2^2a+3d, \ldots, 2^ka+(2^k-1)d)$. We soon realized that our idea extends to the case $A=\left(a,ba+d,b^2a+\frac{b^2-1}{b-1}d,\ldots,b^ka+\frac{b^k-1}{b-1}d\right),$ which was discussed in \cite{LiuXinYeYin}.
We extend the idea further to solve the Frobenius problem for $A=(a,Ha+dB)$ in Theorem \ref{a-2ad-22a3d}.

The same idea also applies to cases like $$A=(a,ha+d,h^2a+(h^2-h+1)d,\ldots,h^ka+(h^k-h+1)d),$$
where $h\geq 2$, $\gcd(a,d)=1$, but solving the corresponding Frobenius problem is still hard. We explain as follows.

Firstly, the \emph{One-Point Theorem} also shows that $(1,h^2-h+1,h^3-h+1,\ldots,h^k-h+1)$ is an orderly sequence,
and we have
\begin{align*}
O_{B}^{H}(M)&=\min\left\{\sum_{i=1}^kh^ix_i \mid \sum_{i=1}^k (h^i-h+1)x_i=M, \ x_i\in\mathbb{N}, 1\leq i\leq k\right\}
\\&=\min\left\{M+(h-1)\cdot \sum_{i=1}^kx_i \mid \sum_{i=1}^k (h^i-h+1)x_i=M, \ x_i\in\mathbb{N}, 1\leq i\leq k\right\}
\\&=M+(h-1)\cdot opt_B(M).
\end{align*}

In the solution of $O_B^H(M)$ mentioned above, by $h(h^i-h+1)+(h^2-2h+1)=h^{i+1}-h+1$, we know that $(x_1,x_2,\ldots,x_k)$ satisfies the following conditions.
\begin{enumerate}
  \item $x_k=\left\lfloor \frac{M}{h^k-h+1}\right\rfloor$.

  \item $0\leq x_1\leq h^2-h$ and $0\leq x_i\leq h$ for every $2\leq i\leq k-1$.

  \item if $2\leq i\leq k-1$ and $x_i=h$, then $0\leq x_1\leq h^2-2h$ and $x_2=\cdots =x_{i-1}=0$.
\end{enumerate}
Similarly, we can define the $R(M)$, a colexicographic order on $R(M)$, and the weight $w(r)=\sum_{i=1}^kh^ix_i$ of $r$.
However, the weight $w(r)$ no longer satisfies an increasing property like the one proven in Lemma \ref{colex-incre}. This makes the Frobenius problem quite difficult to solve.

In \cite{B.Branco}, Branco, Colaco, and Ojeda consider the Frobenius number for a class of generalized repunit numerical semigroups. However, our Collection \texttt{GCNS} (i.e., Collection \ref{e-general-general-model}) does not include this case.

\medskip
\noindent
{\small \textbf{Acknowledgements:}
We are grateful to the anonymous referees for useful comments and suggestions.
The authors would also like to express their sincere appreciation for all the suggestions that have improved the presentation of this paper. This work was partially supported by the National Natural Science Foundation of China [12071311].

\end{document}